\documentclass[12pt]{article}
\usepackage{amsmath}
\usepackage{amsthm}
\usepackage{amssymb}

\newtheorem{theorem}{Theorem}
\newtheorem{corollary}{Corollary}

\newtheorem{lemma}{Lemma}

\begin{document}

\vspace*{30px}

\begin{center}\Large

\textbf{A Family of the Zeckendorf Theorem Related Identities}
\bigskip\bigskip\large

Ivica Martinjak\\
Faculty of Science, University of Zagreb\\
Bijeni\v cka cesta 32, HR-10000 Zagreb, Croatia\\

\end{center}

\begin{abstract} 
In this paper we present a family of identities for recursive sequences arising from a second order recurrence relation, that gives instances of Zeckendorf representation. We prove these results using a special case of an universal property of the recursive sequences. In particular cases we also establish a direct bijection. Besides, we prove further equalities that provide a representation of the sum of $(r+1)$-st and $(r-1)$-st Fibonacci number as the sum of powers of the golden ratio. Similarly, we show a class of natural numbers represented as the sum of powers of the silver ratio. 
\end{abstract}


\section{Introduction}

According to the Zeckendorf's Theorem, every natural number $n$ is uniquely represented as the sum of nonconsecutive Fibonacci numbers $F_k$,
\begin{eqnarray*}
n=F_{k_1}+F_{k_2}+\cdots + F_{k_m}, \enspace k_{i+1} \ge k_{i}+2, \enspace k_i \ge2.
\end{eqnarray*}
Such a sum is called the {\it Zeckendorf representation} of $n$ \cite{Zeck, Zeck2}. The Fibonacci sequence can be naturally extended to negative indexes using the same defining recurrence relation, and terms in this sequence are sometimes called {\it negafibonacci numbers}. D. Knuth has shown that there is a unique representation of an integer $N$ in negafibonacci numbers \cite{Knut}. Representations $$79 =F_{10} + F_{8} + F_{4} = 55 +21 +3$$ and $$-37=F_{-5}+F_{-7}+F_{-10}= 5 +13+(-55)$$ are examples of the former and the latter. Fibonacci identities like
\begin{eqnarray}
4F_n &=& F_{n+2} + F_n + F_{n-2},  \enspace n \ge 2 \label{eqFibo4F} \\
5F_n &=& F_{n+3} + F_{n-1} + F_{n-4}, \enspace n \ge 4 \label{eqFibo5F} \\
6F_n &=& F_{n+3} + F_{n+1} + F_{n-4}, \enspace n \ge 4  \\
11F_n &=& F_{n+4} + F_{n+2} + F_n + F_{n-2} + F_{n-4}, \enspace n \ge 4 
\end{eqnarray} 
gives examples of Zeckendorf representation. According to (\ref{eqFibo5F}), we have 
\begin{eqnarray*}
65&=& F_{10} + F_6 + F_3= 55 + 8 + 2,\\
105& =& F_{11} + F_{7} + F_{4} = 89+13 +3, \ldots
\end{eqnarray*}
Note that the indexes of Fibonacci numbers within these identities are the same as the exponents in the expansion of natural numbers in powers of the golden ratio $\phi$. In particular, $5 = \phi^3 + \phi^{-1} + \phi^{-4}$, $6 = \phi^3 + \phi^{1} + \phi^{-4}$, \ldots. 

This paper aim at finding identities encountering Zeckendorf representation. We also extend these ideas to more general recursive sequences.

\section{Preliminaries}

Given $c_1,c_2, \ldots, c_k \in \mathbb{N}_0$, a $k$-th order {\it linear recurrence} is defined by the recurrence relation 
\begin{eqnarray}\label{eqRecurrenceGeneral}
a_n=c_1a_{n-1}+c_2 a_{n-2} +\cdots + c_ka_{n-k}, \enspace  n \ge k
\end{eqnarray}
and the initial values $a_0, a_1,\ldots, a_{k-1}$. We let $(a_n)_{n \ge 0}$ denote the sequence of numbers $a_0, a_1, \ldots$ defined by this recurrence. The following Lemma \ref{lemTilings} gives combinatorial interpretation of the terms of $(a_n)_{n \ge 0}$ \cite{BeQu}.

\begin{lemma} \label{lemTilings}
Let $a_0=1$. Then $a_n$ is equal to the number of colored {\it tilings} of an {\it n-board} with {\it tiles} of length at most $k$, where a tile of length $i$ can be colored in $c_i$ colors, $1 \le i \le k$. 
\end{lemma}

\begin{proof}
We prove that the number of an $n$-board tilings obey the same recurrence relation as the sequence $(a_n)_{n \ge 0}$. 

We let $a_n'$ denote the number of tilings of an $n$-board. The set of all such $n$-board tilings can be divided into $k$ subsets, where tilings in the $i$-th subset begin with a tile of length $i$, $1 \le i \le k$. The number of tilings in the $i$-th subset is equal to $c_i a'_{n-i}$ which means that the whole set of $n$-board tilings counts $a'_n$ tilings,
\begin{eqnarray*}
a'_n = c_1 a'_{n-1} + c_2 a'_{n-2} + \cdots + c_k a'_{n-k}.
\end{eqnarray*}
From the fact that $a_0=a'_0=1$ it follows $a'_n=a_n$ which completes the proof.
\end{proof}

When $k=2$ relation (\ref{eqRecurrenceGeneral}) reduces to a second order recurrence relation,
\begin{eqnarray} \label{eqSecondOrRec}
u_{n+2}= su_{n+1} + t u_{n}, \enspace  n \ge 0.
\end{eqnarray}
This class of recurrences is of particular interest. According to previous arguments, when $u_0=1$ then $u_n$ represents the number of $n$-board tilings with tiles of length 1 and 2, where these tiles are colored in $s$ and $t$ colors, respectively. Tiles of length 1 are called {\it squares} while tiles of length 2 are called {\it dominoes}.

We use a notion of a {\it breakable} cell of a board tiling when proving the following Lemma \ref{lemMplusN}. It is said that an $n$-board tiling is breakable at cell $m$ if it contains a square at cell $m$ or a domino at cells $m-1$ and $m$. Otherwise it contains a domino covering cells $m$ and $m+1$ and such a tiling is {\it unbreakable} at cell $m$.

\begin{lemma} \label{lemMplusN} 
For the recursive sequence $(u_n)_{n \ge 0}$ and $m \ge0$ we have
\begin{eqnarray} \label{eqMplusN}
u_{m+n}= u_{m}u_n +  t u_{m-1}u_{n-1}.
\end{eqnarray}
\end{lemma}

\begin{proof}
Let consider condition on breakability at cell $m$ of an $(m+n)$-board. We let $A$ denote the set of $(m+n)$-board tilings breakable at cell $m$. On the other hand, let the set $B$ contains those tilings that are unbreakable at cell $m$. Clearly, the set $A$ counts $u_{m}u_{n}$ elements while the set $B$ counts $t u _{m-1} u_{n-1}$ tilings. The fact that
\begin{eqnarray*}
u_{m+n} = |A| + |B|
\end{eqnarray*}
completes the statement of the lemma.
\end{proof}

In what follows in this paper Lemma \ref{lemMplusN} has proved to be very useful.

Two notable representatives of the sequences defined by a second order recurrence (\ref{eqSecondOrRec}) are the Fibonacci sequence and the Pell sequence. We let $(F_n)_{n \ge 0}$ denote the Fibonacci sequence and $(P_n)_{n\ge 0}$ denote the Pell sequence. The Fibonacci sequence arises from (\ref{eqSecondOrRec}) when $s=t=1$ and when initial values are 0 and 1. In other words the sequence of Fibonacci numbers is defined by the recurrence 
\begin{eqnarray} \label{eqFiboDef}
F_{n+2} = F_{n+1} + F_{n}, \enspace F_0=0, \enspace F_1=1.
\end{eqnarray}
Close companions of the Fibonacci sequence are {\it Lucas numbers} $(L_n )_{n \ge 0}$, that are defined by the same recurrence relation but with initial values $L_0=2$ and $L_1=1$. It is worth mentioning that these sequences can also be defined as the only solutions $(x,y)$, $x=L_n$, $y=F_n$ of the Diophantine equation
\begin{eqnarray*}
x^2 - 5y^2 = 4(-1)^n, \enspace n \in \mathbb{N}_0.
\end{eqnarray*} 

There are numerous properties and identities known for the Fibonacci sequence. One can find more in a classic reference on this subject \cite{Vajda}. Recall that the closed form for Fibonacci sequence, called Binet formula, is
\begin{eqnarray}
F_n= \frac{\phi ^n - \bar{\phi} ^n}{\sqrt{5}}
\end{eqnarray}
where $\phi = \frac{1+\sqrt{5}}{2}$ and $\bar{\phi } = \frac{1-\sqrt{5}}{2} $ are solutions of the equation $x^2-x-1=0$. The golden ratio $\phi$ and its conjugate $\bar{\phi}$ are also solutions of the equation $x^n = x^{n-1} + x^{n-2}$ meaning that both values $\phi$ and $\bar{\phi}$ satisfy the Fibonacci recursion (\ref{eqFiboDef}),
\begin{eqnarray*}
\phi^n = \phi^{n-1} + \phi^{n-2}\\
\bar{\phi}^n = \bar{\phi}^{n-1} + \bar{\phi}^{n-2}.
\end{eqnarray*}
There is also recurrence relation encountering the $n$-th power of golden ratio and $n$-th Fibonacci number,
\begin{eqnarray*}
\phi^n =\phi F_n + F_{n-1}.
\end{eqnarray*}

The Pell sequence is defined by the recurrence relation
\begin{eqnarray}
P_{n+2} = P_{n+1} + P_{n}, \enspace P_0=0, \enspace P_1=1.
\end{eqnarray}
The sequence arising from the same recurrence but with initial values 2 and 1 is called {\it Pell-Lucas sequence}. We let $(Q_n)_{n \ge 0}$ denote this sequence. Equivalently, the Pell and Pell-Lucas sequence can be defined as the solutions $(x, y)$, $x=Q_n/2$, $y=P_n$ of the Diophantine equations 
\begin{eqnarray*}
x^2 - dy^2 &=&1\\
x^2 - dy^2 &=&-1
\end{eqnarray*}
when $d=2$. We let $\varphi$ denote the {\it silver ratio}, $\varphi = 1+ \sqrt{2}$, and we let $\bar{\varphi}=1-\sqrt{2}$. Then the closed formula for the $n$-th term in Pell sequence $P_n$ can be written as
\begin{eqnarray}\label{PellClosed}
P_n= \frac{\varphi^n - \bar{\varphi}^n}{ \varphi - \bar{\varphi}}.
\end{eqnarray}

According to the previous arguments, both the Fibonacci and the Pell sequences can be represented as the number of board tilings. However, terms of these sequences start by 0 which means that they are shifted by 1 in comparison to the sequence of related tilings. More precisely, denoting the number of an $n$-board tilings with uncolored squares and dominoes by $f_n$, we have
\begin{eqnarray}
f_n=F_{n+1}.
\end{eqnarray}
Similarly, we let $p_n$ denote the number of tilings of an $n$-board with squares in two colors and uncolored dominoes. The number of such $n$-board tilings is equal to the $(n+1)$-st Pell number, 
\begin{eqnarray}
p_n=P_{n+1}.
\end{eqnarray}
It is worth mentioning that there are numerous results known about the Pell sequence, including basic identities presented in \cite{Bick} and some number properties shown in \cite{Bravo} and \cite{Duje}. On the Zeckendorf representation by Pell numbers one can find more in \cite{Hora} and \cite{Hora2}.

\section{The main result}

We define the sequence $(U_n)_{n\ge 0}$ such that 
\begin{eqnarray}
u_n=U_{n+1}.
\end{eqnarray}

\begin{theorem} \label{thmFamily} For the sequence $(U_n)_{n\ge 0}$,  $r \in \mathbb{N}$, $r \equiv 0 \pmod 2$ and $t=1$
\begin{eqnarray}\label{eqFamily}
(U_{r+1} + U_{r-1} )U_n = U_{n+r} + U_{n-r}.
\end{eqnarray}
\end{theorem}

\begin{proof}

The definition of the sequence $(U_n)_{n\ge 0}$ can be extended to negative indexes. It is obvious that the term having index $-n$ is uniquely determined by that having index $n$,
\begin{eqnarray*}
U_{-n}=(-1)^{n+1} \frac{U_n}{t^n}.
\end{eqnarray*}
Now we apply Lemma \ref{lemMplusN} to both terms in the sum $U_{n+r} + U_{n-r}$,
\begin{eqnarray*}
U_{n+r}+U_{n-r} &=& t U_{n-1}U_r + U_nU_{r+1} + t U_{n-1}U_{-r} + U_nU_{-r+1}\\
    &=&  t U_{n-1}P_r + U_nU_{r+1} + (-1)^{r+1} \frac{ U_{n-1}U_{r} }{t^{n-1}} +(-1)^{r}   \frac{ U_nU_{r-1}}{t^n}\\
   &= & U_nU_{r+1}  +(-1)^r \frac{ U_nU_{r-1} }{t^n} + U_{n-1}U_r \Bigg ( t + \frac{(-1)^{r+1}}{t^{n-1}}  \Bigg).
\end{eqnarray*}
When $t=1$ and $r$ is even the expression above reduces to the first two terms, $U_n U_{r+1} + U_n U_{r-1}$, which completes the proof.
\end{proof}

Since both Fibonacci and Pell sequences satisfy constraint on the coefficient $t$ in Theorem \ref{thmFamily} there are two immediate corollaries of Theorem \ref{thmFamily}.

\begin{corollary}
For the sequence $(F_n)_{n \ge 0}$ of Fibonacci numbers and an even $r \ge2$ we have
\begin{eqnarray} \label{eqFamilyFibo}
(F_{r+1} + F_{r-1} ) F_n = F_{n+r} + F_{n-r}.
\end{eqnarray}
\end{corollary}
The first particular representative of the family of identities (\ref{eqFamilyFibo}) is 
\begin{eqnarray} \label{eqFibo3F}
3F_n &=& F_{n+2} + F_{n-2},
\end{eqnarray}
while further identities are 
\begin{eqnarray}
 7 F_n &=& F_{n+4} + F_{n-4} \label{eqFibo7F} \\
18 F_n &=& F_{n+6} + F_{n-6} \label{eqFibo18F}.
\end{eqnarray}

\begin{corollary} For the sequence $(P_n)_{n \ge 0}$ of Pell numbers and an even $r \ge 2$ we have
\begin{eqnarray} \label{eqPellFamily}
(P_{r+1} + P_{r-1} ) P_n = P_{n+r} + P_{n-r}.
\end{eqnarray}
\end{corollary}
Here we have 
\begin{eqnarray}
6 P_n &=& P_{n+2} + P_{n-2} \label{eqPell6P} \\ 
 34 P_n &=& P_{n+4} + P_{n-4}\\
198 P_n &=& P_{n+6} + P_{n-6} \label{eqPell198P}.
\end{eqnarray}
as the first three particular identities of the family (\ref{eqPellFamily}).

Thus, the first representative of (\ref{eqFamily}) arises when $r=2$,
\begin{eqnarray} \label{eqFamilyFirstCaseGeneral} 
(U_{3} + U_{1} ) U_n = U_{n+2} + U_{n-2}.
\end{eqnarray}
Note that it also can be proved directly using Lemma \ref{lemMplusN},
\begin{eqnarray*}
U_{n+2} + U_{n-2} &=& U_{n-1} U_2 + U_nU_3 + U_{n-1}U_{-2} + U_nU_{-1}\\
  &=&U_n(U_3 + U_1).
\end{eqnarray*}
Moreover, there is a combinatorial proof of this identity and we are going to present it on the instance of Pell sequence. According to the previous definition there are $p_n$ ways to tile an $n$-board with squares in two colors and uncolored dominoes. In order to form $(n+2)$-board tilings ending with
\begin{description}
\item {\it {i)} } a domino,
\item {\it {ii)} } two black squares,
\item {\it {iii)} } black square and white square, respectively,
\item {\it {iv)} } white square and black square, respectively,
\item {\it {v)} } two white squares,
\end{description}
we need five sets of an $n$-board tilings. These $(n+2)$-board tilings are obtained by gluing tiles declared above to the end of $n$-boards in every of these five sets. To complete the set of $(n+2)$-board tilings we need those tilings ending with a square preceded by a domino. This is achieved when we get a set of an $n$-board tilings and insert a domino before a square when appropriate and cut the last domino otherwise. This operation completes the set of $(n+2)$-board tilings and in the same time leave the set of $(n-2)$-board tilings. Clearly, the described procedure of gluing and cut tiles holds in both directions which proves that $$6p_n=p_{n+2} + p_{n-2}.$$

Note that the parameter $r$ within identities  (\ref{eqFibo3F}) - (\ref{eqFibo18F}) is the same as exponents in the expansion of resulting sum $F_{r+1} + F_{r-1}$ in powers of $\phi$,
\begin{eqnarray*}
3= \phi^2 + \phi^{-2},\ldots . 
\end{eqnarray*}
Similarly, we have
\begin{eqnarray*}
6 &=& \varphi^2 + \varphi^{-2},\\ 
34 &=& \varphi^4 + \varphi^{-4}, \ldots
\end{eqnarray*}
where exponents corresponds with the value of parameter $r$  within identities (\ref{eqPell6P}) - (\ref{eqPell198P}). These facts are generalized in the following Theorem \ref{thmFiboGolden} and Theorem \ref{thmPellSilver}.

\begin{theorem} \label{thmFiboGolden} For the Fibonacci sequence $(F_n)_{n \ge 0}$ and $r \in \mathbb{Z}$, $r \equiv 0 \pmod 2$
\begin{eqnarray}
F_{r+1} + F_{r-1} = {\phi}^r + {\phi}^{-r}.
\end{eqnarray}
\end{theorem}

\begin{proof}
Applying the closed formula for Fibonacci sequence to (\ref{eqFamilyFibo}) we get
\begin{eqnarray*}
\Bigg [ \frac{\phi^{r+1} - \bar{\phi}^{r+1}  }{\sqrt{5}}  + \frac{\phi^{r-1} - \bar{\phi}^{r-1}  }{\sqrt{5}} \Bigg ] \Bigg( \frac{\phi^n - \bar{\phi}^n}{\sqrt{5}} \Bigg)&=& \frac{\phi^{n+r} - \bar{\phi}^{n+r}  }{\sqrt{5}}  + \frac{\phi^{n-r} - \bar{\phi}^{n-r}  }{\sqrt{5}}  \\
\Bigg [ \frac{\phi^{r+1} - \bar{\phi}^{r+1}  }{\sqrt{5}}  + \frac{\phi^{r-1} - \bar{\phi}^{r-1}  }{\sqrt{5}} \Bigg ] &=& \frac{\phi^{n+r} - \bar{\phi}^{n+r} + \phi^{n-r} - \bar{\phi}^{n-r}}{\phi^n-\bar{\phi}^n}.
\end{eqnarray*}
Now we have to show that equality 
\begin{eqnarray*}
\frac{\phi^{n+r} - \bar{\phi}^{n+r} + \phi^{n-r} - \bar{\phi}^{n-r}}{\phi^n-\bar{\phi}^n}= \phi^r+ \phi^{-r}
\end{eqnarray*}
holds true. The l.h.s. and r.h.s. of this relation reduces immediately to
\begin{eqnarray*}
- \bar{\phi}^{n+r} - \bar{\phi}^{n-r} &=& - \phi^r\bar{\phi}^n - \bar{\phi}^n \phi^{-r}\\
\bar{\phi}^r +  \bar{\phi}^{-r} &=& {\phi}^r +  {\phi}^{-r}.
\end{eqnarray*}
Finally, we use a property $$-\frac{1}{\phi} = \bar{ \phi}$$ of the golden ratio and its conjugate. For even $r$ we have $$\frac{1}{\phi^r} = \bar{\phi}^r \enspace \Rightarrow \enspace \phi^{-r}= \bar{\phi}^r$$ which means that $\bar{\phi}^r +  \bar{\phi}^{-r} = {\phi}^r +  {\phi}^{-r}$ and completes the statement of the theorem.
\end{proof}

\begin{theorem} \label{thmPellSilver}  For the Pell sequence $(P_n)_{n \ge 0}$ and $r \in \mathbb{Z}$, $r \equiv 0 \pmod 2$
\begin{eqnarray}
P_{r+1} + P_{r-1} = {\varphi}^r + {\varphi}^{-r} .
\end{eqnarray}
\end{theorem}
\begin{proof}
When applying $(\ref{PellClosed})$ to the relation (\ref{eqPellFamily}) we obtain
\begin{eqnarray*}
\Bigg [  \frac{\varphi^{r+1} - \bar{\varphi}^{r+1}}{\varphi - \bar{\varphi}}     + \frac{\varphi^{r-1} - \bar{\varphi}^{r-1}}{\varphi - \bar{\varphi}} \Bigg] = \frac{\varphi^{n+r} - \bar{\varphi}^{n-r}+ \varphi^{n-r} - \bar{\varphi}^{n-r}}{\varphi^n - \bar{\varphi}^n}
\end{eqnarray*}
which means that in order to prove the theorem we have to show that equality 
\begin{eqnarray*}
\frac{\varphi^{n+r} - \bar{\varphi}^{n+r} + \varphi^{n-r} - \bar{\varphi}^{n-r}}{\varphi^n-\bar{\varphi}^n}= \varphi^r+ \varphi^{-r}
\end{eqnarray*}
holds true. Comparison of the l.h.s. and the r.h.s. of this relation gives
\begin{eqnarray*}
- \bar{\varphi}^{n+r}-  \bar{\varphi}^{n-r} &=&  - {\varphi}^{r}  \bar{\varphi}^{n} -  \bar{\varphi}^{n} {\varphi}^{-r}\\
 - \bar{\varphi}^{n}( \bar{\varphi}^{r} +  \bar{\varphi}^{-r} ) &=&  - \bar{\varphi}^{n}(  {\varphi}^{r} +  {\varphi}^{-r} )\\
 {\varphi}^{-r} +  {\varphi}^{r}  &=&    {\varphi}^{r} +  {\varphi}^{-r}.
\end{eqnarray*}
Finally, we employ property that relate the silver ratio and its conjugate
$$
- \frac{1}{1 + \sqrt{2}} = 1 - \sqrt{2}
$$
to show that above relation holds true. This completes the statement of the theorem.
\end{proof}

\section{Some Further Identities}

Using Theorem \ref{thmFamily} and Lemma \ref{lemMplusN} various other identities can be proved. Some of them for the Fibonacci sequence are
\begin{eqnarray}
8 F_n &=& F_{n+4} + F_n + F_{n-4}, \enspace n \ge 4 \\
9 F_n &=& F_{n+4} + F_{n+1} + F_{n-2} + F_{n-4}, \enspace n \ge 4 \\
57F_n &=& F_{n+8} + F_{n+4} + F_{n+2} + F_{n-2} + F_{n-4} + F_{n-8}, \enspace n \ge 8
\end{eqnarray}
and some of them for Pell numbers are 
\begin{eqnarray}
3P_n&=& P_{n+1} + P_{n-1} + P_{n-2}, \enspace n \ge 2\\
4 P_n &=& P_{n+1} + P_{n-1} + P_{n-2}, \enspace n \ge 2\\
20 P_n &=& P_{n+3} + P_{n+2} + P_{n-3} + P_{n-4}, \enspace n \ge 4 \label{eqPell20P} \\
40 P_n &=& P_{n+4} + P_{n+2} + P_{n-2} + P_{n-4}, \enspace n \ge 4.
\end{eqnarray}
For the purpose to prove identity (\ref{eqPell20P}) we have
\begin{eqnarray*}
& &P_{n+3} + P_{n+2} + P_{n-3} + P_{n-4} =\\
 &=& P_{n-1} P_3 + P_nP_4+ P_{n-1}P_2 + P_nP_3 +P_{n-1}P_{-3} + P_nP_{-2}+ P_{n-1} P_{-4} + P_nP_{-3}\\
 &=& 5 P_{n-1} + 12 P_n + 2P_{n-1} + 5 P_n + 5P_{n-1} - 2P_n - 12 P_{n-1} + 5P_n\\
 &=& 20 P_n. 
\end{eqnarray*}
Such identities can also be proved combinatorially, combining arguments presented in proofs of Lemma \ref{lemTilings} and identity (\ref{eqPell6P}).

\end{document}